\documentclass[12pt]{amsart}
\usepackage{graphicx, overpic, epsf,amssymb,amscd,amsfonts,times}

\newtheorem{thm}{Theorem}

\newtheorem{prop}[thm]{Proposition}

\newtheorem{lemma}[thm]{Lemma}
\newtheorem{cor}[thm]{Corollary}

\newcommand{\R}{\mathbb{R}}
\newcommand{\Z}{\mathbb{Z}}
\newcommand{\Q}{\mathbb{Q}}

\newcommand{\bdry}{\partial}

\newcommand{\s}{\vskip.1in}
\newcommand{\n}{\noindent}
\newcommand{\F}{\mathbb{F}}
\renewcommand{\L}{\mathbb{L}}

\newcommand{\M}{\mathbb{M}}
\newcommand{\mo}{\mathcal{M}_{\omega}}

\newcommand{\be}{\begin{enumerate}}
\newcommand{\ee}{\end{enumerate}}

\begin{document}

\title{Giroux torsion and twisted coefficients}

\author{Paolo Ghiggini}
\address{California Institute of Technology, Pasadena, CA 91125}
\email{ghiggini@caltech.edu}
\urladdr{http://www.its.caltech.edu/~ghiggini}

\author{Ko Honda}
\address{University of Southern California, Los Angeles, CA 90089}
\email{khonda@usc.edu} \urladdr{http://rcf.usc.edu/\char126
khonda}

\date{This version: April 9, 2008.}

\keywords{tight, contact structure, Giroux torsion, Heegaard Floer
homology, sutured manifolds, twisted coefficients}

\subjclass{Primary 57M50; Secondary 53C15.}

\thanks{PG was partially supported by a CIRGET fellowship and by the Chaire de Recherche du Canada en
alg\`ebre, combinatoire et informatique math\'ematique de l'UQAM. KH
was supported by an NSF CAREER Award (DMS-0237386).}

\begin{abstract}
We explain the effect of applying a full Lutz twist along a
pre-Lagrangian torus in a contact $3$-manifold, on the contact
invariant in Heegaard Floer homology with twisted coefficients.
\end{abstract}

\maketitle

\s
Let $(M,\xi)$ be a contact $3$-manifold and $T\subset M$ be a {\em
pre-Lagrangian torus}, i.e., an embedded torus whose characteristic
foliation $\xi L=\xi\cap TL$ is linear. By slightly
perturbing $T$, we may assume that it is linearly foliated by
closed orbits, and, by choosing a suitable identification $T\cong
\R^2/\Z^2$, we may assume that the orbits have slope $\infty$.  We
say that $(M,\xi')$ is obtained from $(M,\xi)$ by a {\em full Lutz
twist along $T$} if we cut $(M,\xi)$ along $T$ and insert
$(T^2\times[0,1],\eta_{2\pi })$, where $(x,y,t)$ are coordinates on
$T^2\times [0,1]\cong \R^2/\Z^2\times[0,1]$ and $\eta_{2\pi}=\ker
(\cos (2\pi t) dx - \sin (2\pi t) dy)$.  A contact manifold
$(M,\xi)$ has {\em ${2\pi n}$-torsion along $T$} if there exists a
thickened torus $(T^2\times[0,1],\eta_{2\pi n})$ which embeds into
$(M,\xi)$, so that $T^2\times\{t\}$ are isotopic to $T$ and
$\eta_{2\pi n}$ is obtained by stacking $n$ copies of $\eta_{2\pi}$.
Also $(M,\xi)$ has {\em finite torsion along $T$} if there is a
positive integer $n$ so that $(M,\xi)$ has ${2\pi n}$-torsion along
$T$ but does not have ${2\pi (n+1)}$-torsion along $T$.

In this paper, we assume that our $3$-manifolds are compact and
oriented, and our contact structures are cooriented, unless stated
otherwise. In a previous paper~\cite{GHV}, the authors and Van
Horn-Morris proved the following:

\begin{thm}[Vanishing Theorem]\label{torsion}
Suppose the coefficient ring of the Heegaard Floer homology groups
is $\Z$. If a closed, oriented contact 3-manifold $(M,\xi')$ is
obtained from $(M,\xi)$ by a full Lutz twist along a pre-Lagrangian
torus $T$, then its contact invariant $c(M,\xi';\Z)$ in
$\widehat{HF}(-M;\Z)/\{\pm 1\}$ vanishes.
\end{thm}

Theorem~\ref{torsion} was first conjectured in \cite[Conjecture
8.3]{Gh2}, and partial results were obtained by \cite{Gh1},
\cite{Gh2}, and \cite{LS1}. The corresponding vanishing result for
the contact class in monopole Floer homology is due to
Gay~\cite{Ga}, using results of Mrowka and Rollin~\cite{MR}. Theorem
\ref{torsion}, together with a non-vanishing result of the contact
invariant proved by Ozsv\'ath and Szab\'o \cite[Theorem 4.2]{OSz3},
implies that a contact manifold with $2 \pi$-torsion is not strongly
symplectically fillable. This non-fillability result was conjectured
by Eliashberg, and first proved by Gay \cite{Ga}.

The goal of the present paper is go further and explain what happens
when we use twisted coefficients.  Consider the group ring
$\L=\Z[H_2(M;\Z)]$. If $a \in H_2(M; \Z)$, we denote by $e^a$ the
corresponding element in $\L=\Z[H_2(M; \Z)]$. If $\M$ is a
$\Z$-algebra and $\L\to \M$ is a $\Z$-algebra homomorphism which
induces an $\L$-module structure on $\M$, then the contact invariant
$\underline{c}(M,\xi;\M)$ is an element of
$\underline{\widehat{HF}}(-M;\M)/\M^{\times}$, where $\M^\times$
denotes the group of units of $\M$.\footnote{Although it is stated
in \cite{OSz3} that for any module $\M$ over $\L$ we can get an
element $c(M,\xi;\M)\in \widehat{HF}(-M;\M)/ \L^\times$, in reality
we need to be able to pick out a preferred element (or a collection
of preferred elements) in $\L$ for each intersection point. The
easiest way is for $\M$ to be an $\L$-module via an algebra
homomorphism $\L\to \M$.} In this paper we follow the usual conventions
and underline to indicate that there is a (presumably) nontrivial
$\L$-action on the Heegaard Floer homology or sutured Floer homology
group.

The following is our main theorem:

\begin{thm}\label{thm: main}
There exists a Laurent polynomial $p(t)=t-1\in \Z[t,t^{-1}]$ such
that the following holds: For any closed, oriented contact
$3$-manifold $(M, \xi)$ and a pre-Lagrangian torus $T \subset M$, if
$(M,\xi')$ is obtained from $(M,\xi)$ by a full Lutz twist along
$T$, then
\begin{equation}\label{equation: main}
\underline{c}(M,\xi';\M)= p(e^{[T]}) \cdot \underline{c}(M,\xi;\M).
\end{equation}
Here there is a $\Z$-algebra homomorphism $\L\to \M$ which induces
an $\L$-module structure on $\M$, and $\underline{c}(M,\xi';\M),
\underline{c}(M,\xi;\M)$ are elements in
$\underline{\widehat{HF}}(-M;\M)/\M^\times$.
\end{thm}

Theorem~\ref{thm: main} was partly inspired by the work of
Hutchings-Sullivan~\cite{HS} on the calculation of invariants of
contact structures on the $3$-torus in embedded contact homology.

Observe that, for the applications below, it is only necessary to
know that $p(t)$ is divisible by $t-1$. Let us write $\xi_n$, $n\in
\Z^{\geq 0}$, for the contact structure obtained from $(M,\xi)$ by
applying $n$ full Lutz twists along a pre-Lagrangian torus $T\subset
M$.

\begin{cor}
Let $\L\to \M$ be a $\Z$-algebra homomorphism. If $e^{[T]}$ acts
trivially on $\M$, i.e., as the identity, then
$\underline{c}(M,\xi_n;\M)=0$ in
$\underline{\widehat{HF}}(-M;\M)/\M^\times$ if $n>0$.
\end{cor}

In particular, we have the following:

\begin{cor}
If $T$ is a separating pre-Lagrangian torus in $(M,\xi)$, then for
$n>0$: \be
\item $\underline{c}(M,\xi_n;\M)=0$ in $\underline{\widehat{HF}}(-M;\M)/\M^\times$.
\item $(M,\xi_n)$ is not weakly symplectically fillable.
\ee
\end{cor}

\begin{proof}
For (1), simply observe that $[T]=0$ if $T$ is separating.  (2)
follows from a result of Ozsv\'ath-Szab\'o on the nonvanishing of
the contact invariant for a weakly symplectically fillable contact
structure.  The precise statement is given as Theorem~\ref{thm:
weakly fillable} in Section~\ref{section: preliminaries}.
\end{proof}

On the other hand, Colin~\cite{Co} and
Honda-Kazez-Mati\'c~\cite{HKM1} have proven that there exist
infinitely many nonisomorphic universally tight contact structures
on a toroidal $M$ with a separating torus $T$, of type $(M,\xi_n)$.
This gives large infinite families of universally tight contact
structures which are not weakly symplectically fillable. Our results
generalize prior examples of Ghiggini~\cite{Gh2}.

\begin{cor}
Let $\L=\Z[H_2(M;\Z)]$. If $\underline{c}(M,\xi;\L)\in
\underline{\widehat{HF}}(-M;\L)/\L^\times$ is nontrivial and
non-torsion, i.e., no nonzero element of $\L$ annihilates
$\underline{c}(M,\xi;\L)$, and $[T]\not=0\in H_2(M;\Z)$, then \be
\item $\xi_n$ has finite torsion along $T$ for $n\geq 0$;
\item $\xi_n$ and $\xi_m$ are pairwise nonisotopic for
$n\not=m$. \ee
\end{cor}

\begin{proof}
(2) Write $\L=\Z[t_1,t_1^{-1},\dots,t_k,t_k^{-1}]$,
$t=e^{[T]}=t_1^{a_1}\dots t_k^{a_k}$ and
$c=\underline{c}(M,\xi;\L)$. If $\underline{c}(M,\xi_n;\L)$ and
$\underline{c}(M,\xi_m;\L)$ are equivalent, then there is some
monomial $\pm t_1^{b_1}\dots t_k^{b_k}$ so that
$$\pm t_1^{b_1}\dots t_k^{b_k}(t-1)^nc=(t-1)^mc.$$ Hence
$\pm t_1^{b_1}\dots t_k^{b_k}(t-1)^n=(t-1)^m$, since $c$ is
non-torsion. Given a Laurent polynomial $f=\sum_{i_1,\dots,i_k}
a_{i_1,\dots,i_k} t_1^{i_1}\dots t_k^{i_k}$, its {\em Newton
polytope} is the convex hull of points $(i_1,\dots,i_k)$ in $\R^k$
for which $a_{i_1,\dots,i_k}\not=0$. By comparing the Newton
polytopes of the two polynomials in the above equation, it is
immediate that equality holds if and only if $n=m$.

(1) If $\xi$ has infinite torsion, then there is a sequence of
elements $c_1,c_2,\dots$ in $\L$ so that $c=(t-1)^i c_i$. (Note that
it is a priori not clear whether $c_i=(t-1)c_{i+1}$.) Observe that
$\L$ is Noetherian, since finitely generated polynomial rings over
$\Z$ are Noetherian, and the Noetherian property survives
localization.  Now consider the ascending chain of $\L$-modules:
$$(c)\subset (c_1)\subset (c_1,c_2)\subset
(c_1,c_2,c_3)\subset\dots.$$ By the ascending chain property, the
chain stabilizes at some point, i.e.,
$(c_1,\dots,c_n)=(c_1,\dots,c_{n+1})$.  Hence $c_{n+1}=\sum_{i=1}^n
f_i(t_1,\dots,t_k) c_i$.  Multiplying both sides by $(t-1)^{n+1}$,
we obtain $c=(t-1)g(t_1,\dots,t_k)c$ for some $g(t_1,\dots,t_k)$.
Since $c$ is non-torsion, it follows that $1=(t-1)g(t)$, a
contradiction. The same holds for $\xi_n$.
\end{proof}

\begin{cor}
Suppose $(M,\xi')$ is obtained from $(M,\xi)$ by a full Lutz twist
along $T\subset M$.  If $\underline{c}(M,\xi';\M)\not=0$, then it
follows that $\underline{c}(M,\xi;\M)\not=0$.
\end{cor}

In other words, undoing a full Lutz twist preserves nontriviality of
the contact invariant with twisted coefficients.

\s The main technological advance which allows us to prove
Theorem~\ref{thm: main} without undue effort is the gluing/tensor
product map in sutured Floer homology, proved in \cite{HKM3}.
Sutured Floer homology is an important advance due to Andr\'as
Juh\'asz~\cite{Ju1,Ju2}, and is the relative version of Heegaard
Floer hat theory for a sutured manifold $(M,\Gamma)$. One version of
the main result of \cite{HKM3} is the following: If two sutured
manifolds $(M_1,\Gamma_1)$ and $(M_2,\Gamma_2)$ are glued along
their common boundary (for technical reasons, we assume $\bdry M_1$
and $\bdry M_2$ are connected), then there is a natural tensor
product map:
$$SFH(-M_1,-\Gamma_1;\Z)\otimes_{\Z} SFH(-M_2,-\Gamma_2;\Z)\to
\widehat{HF}(-(M_1\cup M_2);\Z).$$  In Section~\ref{section: SFH and
twisted}, we describe the map in more detail and give a version of
it in the setting of twisted coefficients. Then, in
Section~\ref{section: proof of main}, we prove Theorem~\ref{thm:
main}, modulo the determination of the Laurent polynomial $p(t)$.
Finally, we determine the polynomial $p(t)$ in Section~\ref{section:
determination}.

\section{Preliminaries} \label{section:
preliminaries}

\subsection{Twisted coefficients}
In this subsection we review twisted coefficients with respect to a
closed $2$-form and also the work of Ozsv\'ath and Szab\'o on the
contact class of a weakly fillable contact structure~\cite{OSz3}.
The reader is referred to \cite{OSz2} for the definition and
properties of the Heegaard Floer homology groups with twisted
coefficients.  (Also see \cite{JM} for a good summary which
emphasizes twisted coefficients.)

Let $M$ be a closed $3$-manifold and $[\omega]\in H^2(M;\R)$. Then
$[\omega]$ induces an evaluation map (a group homomorphism):
$$\mbox{$\int$}\colon H_2(M;\Z)\to \R, ~~~~~ [A]\mapsto \mbox{$\int_A \omega$},$$
and we have an induced ring
homomorphism of group rings:
$$\Z[H_2(M;\Z)]\to \Z[\R],$$ which makes $\Z[\R]$
into an $\L=\Z[H_2(M;\Z)]$-module.  We write $\mo$ to indicate
$\Z[\R]$ with this $\L$-module structure.

Given $\mathfrak{t}\in \mbox{Spin}^c(M)$, we can define
$\underline{HF}^\circ(M,\mathfrak{t};\mo)$ for any flavor of
Heegaard Floer homology.  The definition for $\underline{HF}^\infty$
is as follows (and the other $\underline{HF}^\circ$ are analogous):
Let $(\Sigma,\alpha,\beta,z)$ be an admissible pointed Heegaard
diagram for $M$ and let $A$ be a surjective additive assignment for
the Heegaard diagram which takes values in $H_2(M;\Z)$, instead of
the usual $H^1(M;\Z)$. Letting
$\underline{CF}^\infty(M,\mathfrak{t};\mo)$ be the free $\mo$-module
generated by pairs $[\mathbf{x},i]$ with $\mathbf{x}\in
\mathbb{T}_\alpha\cap \mathbb{T}_\beta$ representing $\mathfrak{t}$
and $i\in \Z$, the differential is given by:
$$\underline\bdry^\infty([\mathbf{x},i])=\sum_{\mathbf{y}\in\mathbb{T}_\alpha\cup\mathbb{T}_\beta}
\sum_{\tiny\begin{array}{c} \phi\in\pi_2(\mathbf{x},\mathbf{y})\\
\mu(\phi)=1\end{array}} \# (\mathcal{M}(\phi)/\R) \cdot
t^{\int_{A(\phi)}\omega} [\mathbf{y},i-n_z(\phi)].$$

Next let $X$ be a $4$-dimensional cobordism between $3$-manifolds
$M_0$ and $M_1$, $\omega$ be a closed $2$-form on $X$ defining the
cohomology class $[\omega] \in H^2(X; \R)$, and $\mathfrak{s}$ be a
Spin$^c$ structure on $X$. Then we have maps
\begin{equation} \label{cobordism map} \underline{F}_{X, \mathfrak{s};\mathcal{M}_\omega}^\circ \colon
\underline{HF}^\circ(M_0,\mathfrak{s}|_{M_0}; {\mathcal
M}_{\omega|_{M_0}}) \to
\underline{HF}^\circ(M_1,\mathfrak{s}|_{M_1}; {\mathcal
M}_{\omega|_{M_1}}).\footnote{Our definition of the map looks
slightly different from that given on pp. 323--324 of \cite{OSz3},
but is equivalent to it.}
\end{equation}

In order to define the map for $\underline{HF}^\infty$, we need to
introduce some notation. Let $(\Sigma, {\mathbf \alpha},
\mathbb{\beta}, \mathbb{\gamma}, z)$ be an admissible pointed triple
Heegaard diagram for $X$ so that $(\Sigma, \alpha, \beta,z)$ is a
Heegaard diagram for $M_0$, $(\Sigma, \alpha, \gamma,z)$ is a
Heegaard diagram for $M_1$, and $(\Sigma, \beta, \gamma,z)$ is a
Heegaard diagram for a connected sum of $(S^1 \times S^2)$'s. There
is a canonical intersection point $\mathbf{\Theta} \in
\mathbb{T}_{\beta} \cap \mathbb{T}_{\gamma}$, defined by choosing
the intersection point with lowest relative Maslov index for any
pair of parallel curves $\beta_i$ and $\gamma_i$. We denote by
$\pi_2(\mathbf{x},\mathbf{ \Theta}, \mathbf{y})$ the homotopy
classes of Whitney triangles connecting the intersection points
$\mathbf{x} \in \mathbb{T}_{\alpha} \cap \mathbb{T}_{\beta}$,
$\mathbf{\Theta} \in \mathbb{T}_{\beta} \cap \mathbb{T}_{\gamma}$,
and $\mathbf{y} \in \mathbb{T}_{\alpha} \cap \mathbb{T}_{\gamma}$.
Next, let $A_i$, $i=0,1$, be a surjective additive assignment for
the Heegaard diagram for $M_i$.   Let $\psi_{\mathfrak{s}}\in
\pi_2(\mathbf{x},\mathbf{ \Theta}, \mathbf{y})$ be a fixed
representative of $\mathfrak{s}$. If $\psi \in \pi_2(\mathbf{x'},
\mathbf{\Theta}, \mathbf{y'})$ represents the same ${\rm Spin}^c$
structure $\mathfrak{s}$, then there are Whitney disks
$\phi_{\mathbf{x}',\mathbf{x}}$ and $\phi_{\mathbf{y},\mathbf{y'}}$
so that $\psi= \psi_{\mathfrak{s}} * \phi_{\mathbf{x'},\mathbf{x}} *
\phi_{\mathbf{y},\mathbf{y'}}$, where $*$ denotes concatenation.
(See \cite[Proposition 8.5]{OSz1}.) Then define
$$A_X(\psi)=
\delta(-A_0(\phi_{\mathbf{x'},\mathbf{x}})+A_1(\phi_{\mathbf{y},\mathbf{y'}})),$$
where
$$\delta: H^1(\bdry X)\rightarrow H^2(X,\bdry X)$$
is the coboundary map of the long exact sequence of $(X,\bdry X)$.

We can now define the map
$\underline{F}_{X,\mathfrak{s};\mathcal{M}_\omega}^\infty$ in
Equation~\ref{cobordism map} as follows:
\begin{equation}
\underline{F}_{X,\mathfrak{s};\mathcal{M}_\omega}^\infty
([\mathbf{x}, i]) = \sum_{\mathbf{y} \in \mathbb{T}_{\alpha} \cap
\mathbb{T}_{\gamma}}
\sum_{\tiny \begin{array}{c} \psi \in \pi_2(\mathbf{x}, \mathbf{\Theta}, \mathbf{y}) \\
\psi \mbox{ represents } \mathfrak{s}\\ \mu(\psi)=0
\end{array}} \# {\mathcal M}(\psi) \cdot \ t^{\int_{A_X(\psi)}\omega}
[\mathbf{y}, i-n_z(\psi)].
\end{equation}
Here the map
$\underline{F}_{X,\mathfrak{s};\mathcal{M}_\omega}^\infty$ depends
on the choice of the reference triangle $\psi_{\mathfrak{s}}$, and
changing this choice has the effect of pre-composing (and
post-composing)
$\underline{F}_{X,\mathfrak{s};\mathcal{M}_\omega}^\infty$ by an
element of $H^1(M_0)$ (and an element of $H^1(M_1)$). The
definitions for the other
$\underline{F}_{X,\mathfrak{s};\mathcal{M}_\omega}^\circ$ are
analogous.

Let $\mathfrak{t}_i\in \mbox{Spin}^c(M_i)$ for $i=0,1$.  Let
$S(\mathfrak{t}_0,\mathfrak{t}_1)$ be the set of all Spin$^c$
structures $\mathfrak{s}\in \mbox{Spin}^c(X)$ which restrict to
$\mathfrak{t}_i$ on $M_i$.  Choose a reference Spin$^c$-structure
$\mathfrak{s}_0\in S(\mathfrak{t}_0,\mathfrak{t}_1)$ and choose
$\psi_{\mathfrak{s}}\in
\pi_2(\mathbf{x},\mathbf{\Theta},\mathbf{y})$ for all
$\mathfrak{s}\in S(\mathfrak{t}_0,\mathfrak{t}_1)$, where
$\mathbf{x}$ and $\mathbf{y}$ are the same for all $\mathfrak{s}$.
When summing over $S(\mathfrak{t}_0,\mathfrak{t}_1)$ we form:
$$ \underline{F}_{X,S(\mathfrak{t}_0,\mathfrak{t}_1);\mathcal{M}_\omega}^+=\sum_{\mathfrak{s}\in
\tiny S(\mathfrak{t}_0,\mathfrak{t}_1)}
\underline{F}_{X,\mathfrak{s};\mathcal{M}_\omega}^+
t^{\int_{\mathcal{D}(\psi_{\mathfrak{s}}-\psi_{\mathfrak{s}_0})}\omega}.$$
Here $\mathcal{D}(\psi_{\mathfrak{s}}-\psi_{\mathfrak{s}_0})$ is a
$2$-cycle in $X$ corresponding to the triply-periodic domain
$\psi_{\mathfrak{s}}-\psi_{\mathfrak{s}_0}$.

Now, the composition law \cite[Theorem~3.9]{OSz:triangles} can be
stated as follows, for $\mathcal{M}_\omega$-coefficients:

\begin{thm}[Composition Law] \label{thm: comp law}
Let $X=X_1\cup_{M_1} X_2$ be a composition of cobordisms $X_1$ from
$M_0$ to $M_1$ and $X_2$ from $M_1$ to $M_2$. If $\omega$ is a
closed $2$-form on $X$, then
\begin{eqnarray*}
\underline{F}_{X_2,\mathfrak{s}_2;\mathcal{M}_\omega}^+\circ
\underline{F}_{X_1,\mathfrak{s}_1;\mathcal{M}_\omega}^+
&=&
\sum_{\tiny
\begin{array}{c}
\mathfrak{s}\in\mbox{Spin}^c(X)\\
\mathfrak{s}|_{X_1}=\mathfrak{s}_1,
\mathfrak{s}|_{X_2}=\mathfrak{s}_2
\end{array}}\underline{F}_{X,\mathfrak{s};\mathcal{M}_\omega}^+
t^{\int_{\mathcal{D}(\psi_{\mathfrak{s}}-\psi_{\mathfrak{s}_0})}\omega}
\\
&=& \sum_{\tiny
\begin{array}{c}
\mathfrak{s}\in\mbox{Spin}^c(X)\\
\mathfrak{s}|_{X_1}=\mathfrak{s}_1,
\mathfrak{s}|_{X_2}=\mathfrak{s}_2
\end{array}}\underline{F}_{X,\mathfrak{s};\mathcal{M}_\omega}^+ t^{\langle \omega\cup
(\mathfrak{s}-\mathfrak{s}_0),[X]\rangle},
\end{eqnarray*}
where $\mathfrak{s}_0$ is a reference Spin$^c$ structure on $X$
which restricts to $\mathfrak{s}_1$ on $X_1$ and $\mathfrak{s}_2$ on
$X_2$.\footnote{A related formula is given on p.\ 325 of
\cite{OSz3}, but the term $c_1(\mathfrak{s})$ which appears there
should be replaced by $\mathfrak{s}-\mathfrak{s}_0$.}
\end{thm}

\begin{proof}
The proof follows immediately from \cite{OSz:triangles}, after the
following consideration: Suppose $\psi_{\mathfrak{s}},
\psi_{\mathfrak{s}_0}\in
\pi_2(\mathbf{x},\mathbf{\Theta},\mathbf{y})$ correspond to
Spin$^c$-structures $\mathfrak{s},\mathfrak{s}_0\in
\mbox{Spin}^c(X)$. Then
\[ \int_{\mathcal{D}(\psi_{\mathfrak{s}}-\psi_{\mathfrak{s}_0})}\omega  =
\langle [\omega], PD( \mathfrak{s}-\mathfrak{s}_0) \rangle = \langle
[\omega]\cup(\mathfrak{s}-\mathfrak{s}_0), [X] \rangle,
\]
by the argument in the proof of \cite[Proposition 8.5]{OSz1}.
\end{proof}

The following result is proved in \cite[Theorem~4.2]{OSz3}, using
considerations in the above paragraphs:

\begin{thm}[Ozsv\'ath-Szab\'o] \label{thm: weakly fillable}
Let $(X,\omega)$ be a weak symplectic filling of a contact
$3$-manifold $(M,\xi)$.  Then the contact invariant
$\underline{c}(M,\xi;\mo)$ is nontrivial and non-torsion over $\mo$.
\end{thm}

For our purposes, we are interested in the contact structures
$(T^3,\xi_n)$, $n\in \Z^{\geq 0}$, defined as follows:  Let
$T^3\cong \R^3/\Z^3$ with coordinates $x,y,z$, and let $$\xi_n= \ker
(dz + \varepsilon (\cos (2\pi nz) dx - \sin (2\pi nz) dy)),$$ for
$\varepsilon>0$ small. The contact structures $(T^3,\xi_n)$ can be
weakly filled by $X=D^2\times T^2$ with the product symplectic
structure $\omega=\omega_{D^2}+ dx\land dy$.  Here $\bdry D^2$ is
parametrized by the $z$-coordinate of $T^3$.  Since the pullback of
$\omega$ to $T^3$ is $dx\land dy$, it follows that the image of
$\phi_{[\omega]}$ in $\Z[\R]$ is isomorphic to $\M=\Z[t,t^{-1}]$,
where $[T]$ is the homology class of the torus $dz=const$ and
$t=e^{[T]}$. Hence:

\begin{prop} \label{prop: nonzero}
The contact invariant $\underline{c}(T^3,\xi_n;\M)\in
\underline{\widehat{HF}}(-T^3;\M)$ is nonzero and non-torsion over
$\M$.
\end{prop}

We note that there is a slight difference between $\M$ and $\mo$.
There are two ways around this: either assume $[\omega]$ lives in
$H^2(M;\Q)$ after perturbation, or observe that $\mo$ is a free
$\M$-module. In the latter case, $\underline{c}(T^3,\xi_n;\mo)$ is
the image of $\underline{c}(T^3,\xi_n;\M)\otimes 1$ under the tensor
product map
$$\underline{\widehat{HF}}(-T^3;\M)\otimes_{\M}\mo\to
\underline{\widehat{HF}}(-T^3;\mo),$$ and the nonzero/non-torsion
properties of $\underline{c}(T^3,\xi_n;\mo)$ imply the corresponding
properties for $\underline{c}(T^3,\xi_n;\M)$.

\subsection{Change of coefficients} \label{subsection: change of
coefficients}

We now briefly review the change-of-coefficients spectral sequence,
which will be used extensively throughout this paper. Let $\L$ be a
ring, $\M$ be an $\L$-module, and $(C_*,
\partial)$ be a complex of $\L$-modules. For technical reasons we
will assume that each $C_i$ is a free $\L$-module and there are only
finitely many degrees $i$ for which $C_i$ is nonzero.

The relationship between the homology of $C=(C_*,\partial)$ and the
homology of $C \otimes \M =(C_* \otimes_{\L} \M, \partial \otimes
1)$ is given by the change-of-coefficients spectral sequence.
Consider a free resolution of $\M$:
\[ \ldots \stackrel{f_{n+1}} \longrightarrow {\mathbb F}_n
\stackrel{f_n} \longrightarrow \ldots \stackrel{f_1} \longrightarrow
{\mathbb F}_0  \longrightarrow \M \longrightarrow 0. \]  Then the
double complex $(C_* \otimes_{\L} {\mathbb F}_*, 1 \otimes f_*,
\partial_* \otimes 1)$ gives rise to the spectral sequence
\begin{equation} \label{change-coefficients}
E^2_{i,j}= {\rm Tor}^i_{\L}(H_j(C), \M) \Longrightarrow H_{i+j}(C
\otimes \M),
\end{equation}
where its differentials map $d_k \colon  E^k_{i,j} \to
E^k_{i-k,j+k-1}$. The convergence of the spectral sequence must be
interpreted in the sense that $\bigoplus \limits_{i+j=n}
E^{\infty}_{i,j}$ is the graded module associated to the filtration
on $H_n(C \otimes \M)$ induced by the double complex. For details,
we refer the reader to \cite{Mc}.

\s\n {\bf Example.} Suppose $\L$ is a principal ideal domain (PID).
Then any finitely generated $\L$-module $\M$ is a direct sum whose
summands are of the form $\L/(p)$, $p\in \L$. Hence each $\L/(p)$
admits a free resolution
$$ 0\rightarrow \L\stackrel{p}\rightarrow \L \rightarrow \L/(p)\rightarrow 0,$$
and ${\rm Tor}^i_{\L}(H_j(C), \M) =0$ for all $i \geq 2$. Hence the
$E^2$-term of the spectral sequence consists only of two adjacent
nonzero columns, i.e., the $0$th and $1$st. Since $d_k$ decreases
$i$ by $k$, all differentials $d_k$ with $k\geq 2$ are trivial and $
E^2_{i,j} \cong E^{\infty}_{i,j}$. In this case the convergence of
the spectral sequence means that there is an exact sequence
\begin{equation} \label{eqn: tor for PID}
0 \to H_n(C) \otimes \M  \to H_n(C \otimes \M) \to {\rm
Tor}^1_{\L}(H_{n-1}(C), \M) \to 0.
\end{equation}

\s Returning to the general discussion, observe that there is a
natural map:
\begin{equation} \label{change}
\psi\colon H_i(C) \otimes \M \to H_i(C \otimes \M),
\end{equation}
$$[a]\otimes m\rightarrow [a\otimes m].$$

\begin{lemma} \label{lemma: coefficient-naturality}
Let $\M$ be a ring, whose $\L$-module structure is induced by a ring
homomorphism $\L{} \to \M$. If $(M, \xi)$ is a contact manifold,
then the contact invariant $\underline{c}(\xi; \L) \in
\underline{\widehat{HF}}(-M; \L)/ \L^{\times}$ is mapped to the
contact invariant $\underline{c}(\xi; \M) \in
\underline{\widehat{HF}} (-M; \M)/ \M^{\times}$ by the natural map
$$\underline{\widehat{HF}}(-M; \L) \to \underline{\widehat{HF}} (-M;
\M),$$
$$ [a]\mapsto \psi([a]\otimes 1).$$
\end{lemma}

\begin{proof}
The contact invariant is represented by the same intersection point,
regardless of the coefficient system, and is a cycle for any
coefficient system because there are no holomorphic strips emanating
from it.
\end{proof}

Similar results hold for $\underline{HF}^+$ and for sutured Floer
homology $\underline{SFH}$.

\section{Sutured Floer homology and twisted coefficients}
\label{section: SFH and twisted}

For details on sutured Floer homology and the contact invariant in
sutured Floer homology, the reader is referred to
\cite{Ju1,Ju2,HKM2,HKM3}.

Let $(M,\Gamma)$ be a balanced sutured manifold.  A contact
structure $\xi$ on $M$ with convex boundary and dividing set
$\Gamma$ on $\bdry M$ will be denoted $(M,\Gamma,\xi)$. The contact
invariant of $(M,\Gamma,\xi)$ will be written as
$c(M,\Gamma,\xi;\Z)\in SFH(-M,-\Gamma;\Z)$. Next let
$(M',\Gamma')\subset (M,\Gamma)$ be an inclusion; in particular,
$M'\subset int(M)$. If a connected component $N$ of $M-int(M')$ has
boundary which is not part of $\bdry M'$, then we say $N$ is {\em
not isolated}. Otherwise $N$ is {\em isolated}.

The main result of \cite{HKM3} is the following:

\begin{thm}[Gluing Map] \label{thm: inclusion}
Let $(M',\Gamma')\subset (M,\Gamma)$ be an inclusion, and let $\xi$
be a contact structure on $M-int(M')$ with convex boundary and
dividing set $\Gamma$ on $\bdry M$ and $\Gamma'$ on $\bdry M'$. If
$M-int(M')$ has $m$ isolated components, then $\xi$ induces a
natural map:
$$\Phi_\xi\colon SFH(-M',-\Gamma';\Z)\to SFH(-M,-\Gamma;\Z)
\otimes_{\Z} V^{\otimes m},$$
so that $\Phi_\xi(c(M',\Gamma',\xi';\Z))= c(M,\Gamma,\xi'\cup
\xi;\Z)\otimes (x\otimes\dots\otimes x)$, where $x$ is the contact
class of the standard tight contact structure on $S^1\times S^2$ and
$\xi'$ is any contact structure on $M'$ with boundary condition
$\Gamma'$.  Here $V\cong \widehat{HF}(S^1\times S^2;\Z)\cong
\Z\oplus \Z$ is a $\Z$-graded vector space where the two summands
have grading which differ by one, say $0$ and $1$.
\end{thm}

The sutured Floer homology of $(M,\Gamma)$ can be defined over the
twisted coefficient system $\L=\Z[H_2(M;\Z)]$ --- its definition is
completely analogous to the closed case. If $\L\to \M$ is a
$\Z$-algebra homomorphism, then the invariant for a compact contact
3-manifold $(M,\Gamma,\xi)$ is denoted by
$\underline{c}(M,\Gamma,\xi;\M)\in \underline{SFH}(-M,-\Gamma; \M)/
\M^\times$. If $\bdry M=\emptyset$, then the invariant will be
denoted by $\underline{c}(M,\xi;\M)\in
\underline{\widehat{HF}}(-M;\M)/\M^\times$. (Assuming $\bdry
M=\emptyset$ and $M$ is connected, we view $M$ as the sutured
manifold $(M-B^3,S^1)$, where $B^3$ is a small $3$-ball and $S^1$ is
the suture on $\bdry B^3=S^2$.)

\s The version of Theorem~\ref{thm: inclusion} with respect to
twisted coefficients is the following:

\begin{thm}[Gluing Map, Twisted Coefficients Version]
Let $(M',\Gamma')\subset (M,\Gamma)$ be an inclusion, and let $\xi$
be a contact structure on $M-int(M')$ with convex boundary and
dividing set $\Gamma$ on $\bdry M$ and $\Gamma'$ on $\bdry M'$. Then
there is a natural map
\begin{equation}
\Phi_\xi\colon \underline{SFH}(-M',-\Gamma';\Z[H_2(M')])\rightarrow
\underline{SFH}(-M,-\Gamma;\Z[H_2(M)]), \end{equation} so that
$\Phi_\xi(\underline{c}(M',\Gamma',\xi';\Z[H_2(M')])=\underline{c}(M,\Gamma,\xi'\cup\xi;\Z[H_2(M)]),$
where $\xi'$ is a contact structure on $M'$ with boundary condition
$\Gamma'$.
\end{thm}

Here, if $(M',\Gamma')=(M_1',\Gamma_1')\sqcup (M_2',\Gamma_2')$,
then $\underline{SFH}(-M',-\Gamma';\Z[H_2(M')])$ is isomorphic to
$$\underline{SFH}(-M_1',-\Gamma_1';\Z[H_2(M_1')])\otimes_{\Z}
\underline{SFH}(-M_2',-\Gamma_2';\Z[H_2(M_2')]).$$

\begin{proof}[Sketch of Proof.]
We briefly explain the modifications needed for the proof in the
twisted coefficients case.

Without loss of generality, consider the situation where we glue
$(M_1',\Gamma_1')$ and $(M_2',\Gamma_2')$ along a common closed,
oriented, connected surface $T_0$ (so that the sutures match) to
obtain $(M,\Gamma)$. More precisely, suppose the following holds:
$(M',\Gamma')=(M_1',\Gamma_1')\sqcup (M_2',\Gamma_2')$, $M_1',M_2'$
are connected, and each $M_i'$ has more than one boundary component.
Moreover, $M'\subset int(M)$ so that $M-int(M')$ consists of
components $T_i\times[0,1]$, $i=0,1,\dots, k$, where $T_i$ are
closed, oriented, connected surfaces and the contact structures on
$T_i\times[0,1]$ are $[0,1]$-invariant and compatible with the
dividing set on $\bdry M\sqcup \bdry M'$. The component
$T_0\times[0,1]$ has one boundary component $T_0\times\{0\}\subset
\bdry M_1'$ and the other boundary component $T_0\times\{1\}\subset
\bdry M_2'$. Each $T_i\times[0,1]$, $i=1,\dots,k$, has one boundary
component $\subset \bdry M$ and the other boundary component
$\subset \bdry M'_j$ for some $j$.

Let $\Sigma'$ be a compatible Heegaard surface for $(M',\Gamma')$,
and $\Sigma$ be an extension to a Heegaard surface for $(M,\Gamma)$,
as given by \cite{HKM3}. In particular, $\Sigma$ is
contact-compatible on $T_i\times[0,1]$, $i=0,\dots,k$. Since there
is one isolated component $T_0\times[0,1]$, the sets of
$\alpha'$-curves and $\beta'$-curves for $\Sigma'$ cannot be
extended to a complete set of $\alpha$-curves and $\beta$-curves for
$\Sigma$. In order to remedy this problem, we take a connected sum
of $M$ with $S^1\times S^2$. More precisely, on the
contact-compatible portion, $\Sigma$ is (locally) of the form $\bdry
(S\times[0,1])$, where $S$ is a surface with boundary (i.e., a page
of a very partial open book) which may possibly be disconnected.
Then we attach a $2$-dimensional $1$-handle to $S$ so as to connect
$T_0\times[0,1]$ to some other $T_i\times[0,1]$ adjacent to $M_2'$.
(On the level of $\Sigma$, we remove two disks and glue their
boundaries together.)  This gives rise to a Heegaard decomposition
$(\Sigma'',\alpha'',\beta'')$ for $M''=M\#(S^1\times S^2)$.


The inclusion $M'\hookrightarrow M''$ gives rise to a group
homomorphism $H_2(M')\rightarrow H_2(M'')$ and also to an algebra
homomorphism $\Z[H_2(M')]\rightarrow \Z[H_2(M'')].$ Here
$$\Z[H_2(M')]=\Z[H_2(M'_1)\oplus H_2(M'_2)]\cong
\Z[H_2(M'_1)]\otimes_{\Z} \Z[H_2(M'_2)].$$ Therefore, the group
$\underline{SFH}(-M'',-\Gamma;\Z[H_2(M'')])$ has the structure of a
$\Z[H_2(M'')]$-module and also of a $\Z[H_2(M'_1)]\otimes_{\Z}
\Z[H_2(M'_2)]$-module.

By the above two paragraphs, the inclusion of
$(\Sigma',\alpha',\beta')$ into $(\Sigma'',\alpha'',\beta'')$ gives
rise to the map
\begin{equation} \Phi_\xi\colon
\underline{SFH}(-M',-\Gamma';\Z[H_2(M')])\rightarrow
\underline{SFH}(-M'',-\Gamma;\Z[H_2(M'')]),
\end{equation}
obtained by tensoring with the contact class in the
contact-compatible portion. The fact that $\Phi_\xi$ is independent
of the choices is proved in the same way as in Theorem~\ref{thm:
inclusion} and will be omitted.

Finally, we claim that:
\begin{equation} \label{eqn: same}
\underline{SFH}(-M'',-\Gamma;\Z[H_2(M'')])\cong
\underline{SFH}(-M,-\Gamma; \Z[H_2(M)]),
\end{equation}
where the
isomorphism is a $\Z[H_2(M'')]$-module homomorphism, given by the
projection map
$$\pi_1\colon H_2(M'')\cong H_2(M)\oplus H_2(S^1\times S^2)\to H_2(M)$$
onto the first factor.  Here $e^{[\{pt\}\times S^2]}$ acts
trivially. This is a simple generalization of the calculation of
$$\underline{\widehat{HF}}(S^1\times S^2; \Z[H_2(S^1\times
S^2)])\cong \underline{\widehat{HF}}(S^1\times S^2; \Z[t,t^{-1}]),$$
where $t$ is the exponential of the homology class of $\{pt\}\times
S^2$ (or, equivalently, $[\{pt\}\times S^2]$, viewed
multiplicatively). The chain complex is generated by two generators
$x,y$, and $\bdry x=0$, $\bdry y= (t-1)x$. Therefore,
$$\underline{\widehat{HF}}(S^1\times S^2; \Z[H_2(S^1\times S^2)])\cong
\Z[t,t^{-1}]/(t-1)\cong \Z.$$ The homology group is generated by
$x$, and $t$ acts trivially (i.e., by the identity) on $x$.
\end{proof}

\s\n {\bf Example.} Consider $M_1'=T^2\times[0,1]$ and
$M_2'=T^2\times[1,2]$. Let $t_i$ be the exponential of the generator
of $H_2(M_i')$, $i=1,2$. If we glue to obtain $M=T^2\times[0,2]$,
then the map
\begin{eqnarray*}
\Phi\colon\underline{SFH}(-M_1',-\Gamma_1';\Z[H_2(M_1')])\otimes_{\Z}
\underline{SFH}(-M_2',-\Gamma_2';\Z[H_2(M_2')])\hspace{.8in}\\
\hspace{.8in}\to \underline{SFH}(-M'',-\Gamma;\Z[H_2(M'')])\cong
\underline{SFH}(-M,-\Gamma;\Z[H_2(M)]),
\end{eqnarray*}
is a $\Z[t_1,t_1^{-1}]\otimes_{\Z}\Z[t_2,t_2^{-1}]$-module
homomorphism.  Since $e^{[\{pt\}\times S^2]}$ acts trivially, it
follows that the multiplication by $t_1$ and $t_2$ are the same.
Hence the above becomes a homomorphism in the category of
$\Z[t,t^{-1}]$-modules:
$$\underline{SFH}(-M_1',-\Gamma_1';\Z[t,t^{-1}])\otimes_{\Z}
\underline{SFH}(-M_2',-\Gamma_2';\Z[t,t^{-1}])\to
\underline{SFH}(-M,-\Gamma;\Z[t,t^{-1}]).$$

\section{Proof of Theorem~\ref{thm: main}}
\label{section: proof of main}

In this section we prove Theorem~\ref{thm: main} without precisely
determining the Laurent polynomial $p(t)\in \Z[t,t^{-1}]$. The
precise polynomial will be determined in Section~\ref{section:
determination}.

Let $\Gamma$ be the following suture/dividing set on the boundary of
$N=T^2\times[0,1]$: $\#\Gamma_{T_0}=\#\Gamma_{T_1}=2$,
$\Gamma_{T_0}$ and $\Gamma_{T_1}$ have no homotopically trivial
components, $\mbox{slope}(\Gamma_{T_0})= 0$, and
$\mbox{slope}(\Gamma_{T_1})=\infty$. Here $\#$ denotes the number of
connected components, $T_i=T^2\times\{i\}$, and the orientation of
$T_i$ is inherited from that of $T^2$. (Hence $\bdry N=T_1\cup
-T_0$.)

Next, let $[T]\in H_2(N)$ be the generator representing
$T^2\times\{pt\}$.  Also let $\L=\Z[H_2(N)]=\Z[t,t^{-1}]$, where
$t=e^{[T]}$. Then we have the following:

\begin{lemma}
If $\L\to \M$ is a $\Z$-algebra homomorphism, then \be
\item[(i)] $\underline{SFH}(N,\Gamma;\M)\cong \M\oplus \M\oplus \M \oplus
\M,$ where each direct summand represents a distinct Spin$^c$
structure; \item[(ii)] If $(N,\Gamma,\xi)$ is a basic slice, then
$\underline{c}(N,\Gamma,\xi;\M)$ generates the appropriate $\M$. \ee
\end{lemma}

\begin{proof}
This follows from \cite[Section~5, Example~4]{HKM2}, as well as
\cite[Figure~15]{HKM2}.  There are four intersection points in the
Heegaard diagram given in that figure, and no holomorphic disks
between any two. Hence each intersection point generates an
$\M$-direct summand. Moreover, one of the intersection points
corresponds to the contact invariant for a basic slice.
\end{proof}

\begin{lemma} \label{lemma: add to basic slice}
If $(N,\Gamma,\xi)$ is a basic slice, then $(N,\Gamma,\xi')$,
obtained from $\xi$ via a full Lutz twist along a pre-Lagrangian
torus parallel to $T^2\times\{pt\}$, satisfies:
\begin{equation}\label{equation: add torsion}
\underline{c}(N,\Gamma,\xi';\L)=p(t)\cdot
\underline{c}(N,\Gamma,\xi;\L),
\end{equation}
where $p(t)$ is a nonzero element in $\L$ which satisfies $p(1)=0$.
\end{lemma}

\s\n {\em Remark.} Such a pre-Lagrangian torus exists by
\cite[Corollary~4.8]{H1}.

\begin{proof}
Since $\underline{c}(N,\Gamma,\xi;\L)$ generates $\L$, and $\xi$ and
$\xi'$ are homotopic (hence are in the same Spin$^c$ structure), it
follows that there is some element $p(t)$ of $\L$ so that
Equation~\ref{equation: add torsion} holds.

Next we apply Lemma~\ref{lemma: coefficient-naturality}, i.e., the
naturality of the contact invariant with respect to change of
coefficients. Consider the algebra homomorphism $\L\rightarrow
\Z=\L/(t-1)$.  The corresponding map
$$\underline{SFH}(-N,-\Gamma;\L)\rightarrow SFH(-N,-\Gamma;\Z)$$
sends $c(N,\Gamma,\xi;\L)$ to $c(N,\Gamma,\xi;\Z)$ and
$c(N,\Gamma,\xi';\L)$ to
$c(N,\Gamma,\xi';\Z)=p(1)c(N,\Gamma,\xi;\Z)$. Now, since
$c(N,\Gamma,\xi;\Z)\not=0$ and we know from \cite{GHV} that
$c(N,\Gamma,\xi';\Z)=0$, it follows that $p(1)=0$.

To prove that $p(t)$ is nonzero, we show that
$\underline{c}(N,\Gamma,\xi';\L)\not=0$.  In fact, there is an
inclusion of $(N,\Gamma)$ into $T^3$ which sends $\xi'$ to $\xi_2$,
the double cover of the standard Stein fillable contact structure on
$T^3$. We have a corresponding inclusion map
$$\Phi\colon \underline{SFH}(-N,-\Gamma;\Z[H_2(N)])\to
\underline{\widehat{HF}}(-T^3;\Z[H_2(T^3)]).$$ Taking a basis
$\{T,T',T''\}$ for $H_2(T^3)$ where $T$ comes from $H_2(N)$, and
setting $T'=T''=0$, we have a projection of $H_2(T^3)$ to $\Z$,
generated by $T$. This gives rise to:
$$\Phi\colon \underline{SFH}(-N,-\Gamma;\L)\to
\underline{\widehat{HF}}(-T^3;\L),$$ which sends
$\underline{c}(N,\Gamma,\xi';\L)$ to $\underline{c}(T^3,\xi_2;\L)$.
Now, $\underline{c}(T^3,\xi_2;\L)$ is nonzero by
Proposition~\ref{prop: nonzero}. This implies that
$\underline{c}(N,\Gamma,\xi';\L)$ is also nonzero.
\end{proof}

Now, $p(1)=0$ means that $p(t)$ is divisible by $t-1$. This function
$p(t)$ is now the universal Laurent polynomial which is multiplied
whenever $2\pi$-torsion is added.

\begin{proof}[Completion of proof of Theorem~\ref{thm: main} without
determining $p(t)$.] Let $T$ be the pre-Lagrangian torus, along
which the full Lutz twist will be applied.  Then there exists a
basic slice $(N_1=T^2\times[0,1],\xi|_{N_1})\subset (M,\xi)$ so that
$T\subset N_1$ (and $T$ is parallel to $T^2\times\{t\}$). Decompose
$M$ into $N_1$ and $N_2=M-N_1$. Moreover,
$\underline{c}(M,\xi;\Z[H_2(M)])$ is obtained by taking the tensor
product of $\underline{c}(N_i,\xi|_{N_i}; \Z[H_2(N_i)])$, $i=1,2$.
Now, applying a full Lutz twist is equivalent to changing
$\xi|_{N_1}$ to $\xi|_{N_1}'$ as in Lemma~\ref{lemma: add to basic
slice}.  This has the effect of multiplying
$\underline{c}(N_1,\xi|_{N_1}; \L)$ by $p(t)$. The theorem now
follows from linearity.
\end{proof}

\section{Determination of $p(t)$} \label{section: determination}

The goal of this section is to prove the following theorem:

\begin{thm}\label{thm: polynomial}
$p(t)=t-1$.
\end{thm}

To pin down the polynomial we compute an example. A natural
candidate is $T^3$, whose tight contact structures $\xi_n$ are all
obtained by applying $(n-1)$ full Lutz twists to the Stein fillable
contact structure $\xi_1$.

\subsection{Calculation of $\underline{HF}^+(T^3;\M)$}
Fix a primitive cohomology class $[\omega] \in H^2(T^3;\Z)$ and
define the $\Z[H_2(T^3;\Z)]$--module $\M$ as $\Z[t,t^{-1}]$ endowed
with the $H_2(T^3;\Z)$-action $c \cdot t=t^{\langle \omega, c
\rangle}$. In this subsection we calculate
$\underline{HF}^+(T^3;\M)$ and determine the location of the contact
invariant $\underline{c}(T^3,\xi_n;\M)$. The computation in this
subsection is similar to that of \cite[Proposition~8.4]{OSz4}.

We will only be interested in the Spin$^c$ structure
$\mathfrak{s}_0$ which satisfies $c_1(\mathfrak{s}_0)=0$. The
Heegaard Floer homology classes for the other Spin$^c$ structures
vanish by the adjunction inequality.

\begin{lemma}\label{lemma: part0}
$\underline{HF}^{\infty}_j(T^3;\M) \cong \Z^2$ for all half-integers
$j$.
\end{lemma}

\begin{proof}
In \cite[Theorem~10.12]{OSz2}, it was shown that
$$\underline{HF}^{\infty}(T^3;\Z[H_2(T^3;\Z)])\cong \Z[U,U^{-1}],$$
where $U$ has degree $-2$. Also, from the computation of
$\underline{HF}^+(T^3;\Z[H_2(T^3;\Z)])$ in \cite[Section 8.4]{OSz4},
one easily sees that the nonzero elements sit in degrees $j \equiv
\frac 12$ mod $2$. In order to prove our lemma, we apply the
change-of-coefficients spectral sequence
\[ {\rm Tor}^i_{\Z[H_2(T^3;\Z)]}(\underline{HF}_j^{\infty}
(T^3;\Z[H_2(T^3;\Z)]),\M) \Rightarrow
\underline{HF}^{\infty}_{i+j}(T^3; \M). \]

Choose coordinates on $H_2(T^3;\Z)$ so that the group algebra
$\Z[H_2(T^3;\Z)]$ is identified with $\L=\Z[t_1, t^{-1}_1, t_2,
t^{-1}_2, t_3, t^{-1}_3]$ and $\M$ is identified with $\Z[t_3,
t^{-1}_3]$ on which $t_1$ and $t_2$ act trivially. In order to
compute the $E^2$-term of the spectral sequence, we take a free
resolution of $\Z[t_3, t^{-1}_3]$:
\begin{equation} \label{eqn: flat}
0 \longrightarrow \L\stackrel{f} \longrightarrow \L\oplus \L
\stackrel{g} \longrightarrow \L\longrightarrow \Z[t_3, t^{-1}_3]
\longrightarrow 0,
\end{equation} where
\[ f(1)= (t_2-1,t_1-1),
\quad g(1,0)= t_1-1, \quad \text{and} \quad g(0,1)= t_2-1. \]
Observe that $\operatorname{Im}(f)=\operatorname{Ker}(g)$ follows
from the fact that $\L$ is a unique factorization domain.

If we truncate the last term in Equation~\ref{eqn: flat} and tensor
with $\Z$ over $\L$, then we obtain the complex
\[ 0 \longrightarrow \Z \longrightarrow \Z \oplus
\Z \longrightarrow \Z \longrightarrow 0,
\] where all maps are trivial. Thus

\[ {\rm Tor}^i_{\Z[H_2(T^3;\Z)]}(\underline{HF}^{\infty}_j
(T^3;\Z[H_2(T^3;\Z)]),\M) \cong \left\{
\begin{array}{ll}
\Z, & \mbox{if } i=0;\\
\Z^2, & \mbox{if } i=1;\\
\Z, & \mbox{if } i=2\\
\end{array}
\right. \]
if $j \equiv \frac 12$ mod $2$, and $0$ otherwise.

Therefore the $E^2$-term of the spectral sequence has the form
\[ \begin{matrix}
\vdots & \vdots & \vdots \\
\Z & \Z^2 & \Z \\
0 & 0 & 0 \\
\Z & \Z^2 & \Z \\
0 & 0 & 0 \\
\vdots & \vdots & \vdots
\end{matrix} \]
and all the higher differentials are trivial for degree reasons, so
$E^2 \cong E^{\infty}$. For $i$ odd and $j \equiv \frac 12$ mod $2$
we have $\underline{HF}^{\infty}_{i+j}(T^3; \M) \cong \Z^2$, and for
$i$ even and $j \equiv \frac 12$ mod $2$ the $E^{\infty}$-term gives
an exact sequence
\[ 0 \longrightarrow \Z \longrightarrow \underline{HF}^{\infty}_{i+j}
(T^3; \M) \longrightarrow \Z \longrightarrow 0, \] so
$\underline{HF}^{\infty}_{i+j}(T^3; \M) \cong \Z^2$ in this case
also, since $\Z$ is a free Abelian group.
\end{proof}

Let $M\{ a,b,c \}$ denote the $3$-manifold obtained by surgery on
the Borromean rings with surgery coefficients $a$, $b$, and $c$. The
$3$-torus $T^3$ is homeomorphic to $M\{ 0,0,0 \}$.

\begin{lemma} \label{lemma: part1}
$\underline{HF}^+_j(T^3;\M) \cong \underline{HF}^{\infty}_j(T^3;\M)$
for all $j \geq \frac 12$, and $\underline{HF}^+_j(T^3;\M) =0$ for
all $j \leq - \frac 32$.
\end{lemma}

\begin{proof}
From \cite[Theorem 9.21]{OSz2}, we have the exact sequence:
\begin{equation} \label{eqn: borromean}
\ldots \longrightarrow \widehat{HF}(M \{ 0,0,\infty \};\Z)
[t,t^{-1}] \longrightarrow \underline{\widehat{HF}} (M \{ 0,0,0
\};\M)\longrightarrow
\end{equation}
\begin{equation*}
\longrightarrow \widehat{HF}(M \{ 0,0,1 \};\Z)[t,t^{-1}]
\longrightarrow \ldots,
\end{equation*}
where the two central maps decrease the degree by $\frac 12$ (see
\cite[Lemma 3.1]{OSz4}).  Now, according to the proof of
\cite[Proposition 8.4]{OSz4}, $$\widehat{HF}(M \{ 0,0,1 \};\Z)\cong
\left\{
\begin{array}{ll}
\Z^2, & \mbox{if } j=-1,0;\\
0, & \mbox{otherwise}. \end{array} \right.$$  Also
$M\{0,0,\infty\}=(S^1\times S^2)\# (S^1\times S^2)$, so
$$\widehat{HF}_j(M\{0,0,\infty\};\Z)\cong \left\{
\begin{array}{ll}
\Z, & \mbox{if } j=1;\\
\Z^2, & \mbox{if } j=0;\\
\Z, & \mbox{if } j=-1;\\ 0, & \mbox{otherwise}.
\end{array}\right.$$ Therefore, by Equation~\ref{eqn: borromean},
$\underline{\widehat{HF}}(T^3;\M)$ is supported in degrees ${1\over
2}, -{1\over 2}, -{3\over 2}$.

We now claim that $\underline{\widehat{HF}}_{-{3\over
2}}(T^3;\M)=0$. Suppose on the contrary that
$\underline{\widehat{HF}}_{-{3\over 2}}(T^3;\M)\not=0$. If we forget
the action of $t$ and view $\underline{\widehat{CF}}(T^3;
\Z[t,t^{-1}])$ as a $\Z$-module, then it is easy to see that there
is some prime $p$ for which $\underline{\widehat{HF}}_{-{3\over
2}}(T^3; \F_p[t,t^{-1}]) \not=0$, where $\F_p$ is the field of $p$
elements. The advantage of $\F_p[t,t^{-1}]$ is that it is a PID.
Now, if $\mathbb{K}$ is an $\F_p[t,t^{-1}]$-module, then we can
apply Equation~\ref{eqn: tor for PID} to obtain
$$\underline{\widehat{HF}}_{- \frac 32}(T^3; {\mathbb K}) \cong
\underline{\widehat{HF}}_{- \frac 32}(T^3; \F_p[t,t^{-1}])
\otimes_{\F_p[t,t^{-1}]} {\mathbb K},$$ in view of the fact that
$\underline{\widehat{HF}}_{- \frac 52}(T^3; \F_p[t,t^{-1}])=0$. If
${\mathbb K}$ is a field, then the existence of an
orientation-reversing diffeomorphism of $T^3$ forces
$\underline{\widehat{HF}}_{- \frac 32}(T^3; {\mathbb K}) \cong
\underline{\widehat{HF}}_{\frac 32}(T^3; {\mathbb K}) =0$ (see
\cite[Proposition 7.11]{OSz:triangles}). Since $\F_p[t,t^{-1}]$ is a
PID, we can decompose
\[ \underline{\widehat{HF}}_{- \frac 32}(T^3; \F_p[t,t^{-1}])
\cong \F_p[t,t^{-1}]^{n_0} \oplus \F_p[t,t^{-1}]/ (f_1^{n_1}) \oplus
\ldots \oplus \F_p[t,t^{-1}]/(f_k^{n_k}), \] where the $f_i$ are
irreducible. In particular, if we take ${\mathbb K}=
\F_p[t,t^{-1}]/(f_1)$, then we have $$\underline{\widehat{HF}}_{-
\frac 32}(T^3; \F_p[t,t^{-1}]) \otimes_{\F_p[t,t^{-1}]} {\mathbb K}
\neq 0,$$ which is a contradiction. This proves the claim.

Next, from the exact sequence $$ \ldots \longrightarrow
\underline{\widehat{HF}}_{k+2}(T^3;\M) \longrightarrow
\underline{HF}^+_{k+2}(T^3;\M) \stackrel{U} \longrightarrow
\underline{HF}^+_k(T^3;\M) \longrightarrow$$
$$\longrightarrow \underline{\widehat{HF}}_{k+1}(T^3;\M)
\longrightarrow \ldots,$$ we see that $U$ is an isomorphism if $k
\geq \frac 12$ or if $k \leq - \frac 72$. Since $HF^+$ is isomorphic
to $HF^{\infty}$ in sufficiently high degrees and zero in
sufficiently low degrees, the lemma follows.
\end{proof}

\begin{lemma} \label{lemma: part2}
$\underline{HF}^+_{- \frac 12}(T^3;\M) \cong \Z^2 \oplus
\Z[t,t^{-1}]$.
\end{lemma}

\begin{proof}
First recall that $$HF^+(\#^2(S^1\times S^2);\Z)\cong \Lambda^*
H^1(\#^2(S^1\times S^2);\Z)\otimes \Z[U^{-1}],$$ where $U$ has
degree $-2$. This means that
$$HF^+_j(\#^2(S^1\times S^2);\Z)= \left\{
\begin{array}{ll}
\Z^2, & \mbox{if } j\geq 0;\\
\Z, & \mbox{if } j=-1;\\
0, & \mbox{if } j<-1. \end{array} \right.$$  Next, we calculate that
$$HF^+_j(M\{0,0,1\};\Z)\cong \left\{
\begin{array}{ll}
\Z^2, & \mbox{if } j\geq -1;\\
0, & \mbox{if } j<-1. \end{array} \right.$$ This follows immediately
from
$$\widehat{HF}_j(M\{0,0,1\};\Z)\cong \left\{
\begin{array}{ll}
\Z^2, & \mbox{if } j=0,-1;\\
0, & \mbox{otherwise}.\end{array} \right.$$ and
$HF^\infty_j(M\{0,0,1\};\Z)\simeq \Z^2$ for all $j$
(\cite[Theorem~10.1]{OSz2}).

Now we apply the surgery exact sequence for the triple
$M\{0,0,\infty\}$, $M\{0,0,0\}$, $M\{0,0,1\}$ with twisted
coefficients:
\[ \dots\longrightarrow (\Z[t,t^{-1}])^2 \stackrel{\underline{F_1}^+} \longrightarrow
\underline{HF}^+_{- \frac 12}  (T^3;\M) \stackrel{\underline{F_2}^+}
\longrightarrow  (\Z[t,t^{-1}])^2 \stackrel{\underline{F_3}^+}
\longrightarrow
\Z[t,t^{-1}]\stackrel{\underline{F_1}^+}\longrightarrow 0. \] The
image of $\underline{F_1}^+$ is isomorphic to $\Z^2$ by
$U$--equivariance. Indeed, for $j\geq {1\over 2}$, the long exact
sequence splits as:
$$0\to (\Z[t,t^{-1}])^2\to
(\Z[t,t^{-1}])^2\to \underline{HF}^+_j(T^3;\M)\cong
\Z^2\to 0,$$
$$0\to (\Z[t,t^{-1}])^2\to
(\Z[t,t^{-1}])^2\to \underline{HF}^+_{-{1\over
2}}(T^3;\M)\to \dots,$$ and we can apply $U$ to the top
sequence when $j={3\over 2}$. We now claim that
$\operatorname{Im}(\underline{F_2}^+)=\operatorname{Ker}(\underline{F_3}^+)$
is isomorphic to $\Z[t,t^{-1}]$.  Indeed, since $\underline{F_3}^+$
is surjective, $\underline{F_3}^+$ must map $(1,0)\mapsto f(t)$,
$(0,1)\mapsto g(t)$, where $f,g$ are relatively prime and hence have
no common factors. By unique factorization on $\Z[t,t^{-1}]$,
$\operatorname{Ker}(\underline{F_3}^+)$ must be generated by
$(g(t),f(t))$ and is free since there are no zero divisors. The
sequence splits since $\operatorname{Im}(\underline{F_2}^+)$ is
free.
\end{proof}

Putting together Lemmas~\ref{lemma: part0}, \ref{lemma: part1}, and
\ref{lemma: part2}, we obtain:

\begin{prop} \label{prop: summary} $\mbox{}$
$$\underline{HF}^+_j(T^3;\M)\cong \left\{
\begin{array}{ll}
\Z^2, & \mbox{if } j \geq {1\over 2};\\
\Z^2 \oplus \Z[t,t^{-1}], & \mbox{if } j=-{1\over 2};\\
0, & \mbox{if } j \leq - {3\over 2}.
\end{array} \right.$$
\end{prop}

Next we identify the location of the contact invariant for $\xi_n$
on $T^3$.

\begin{lemma} \label{lemma: where is contact}
The contact invariant of $(T^3,\xi_n)$ has a non-zero component in
the summand $\Z[t,t^{-1}]$ of $\underline{HF}^+_{-{1\over
2}}(T^3;\M)$.
\end{lemma}

\begin{proof}
We claim that $\underline{HF}_{red}(T^3;\M)$ can be identified with
the summand of $\underline{HF}^+_{-{1\over 2}}(T^3;\M)$ isomorphic
to  $\Z[t,t^{-1}]$. For this we use the exact triangle
$$\dots\to \underline{HF}^\infty_j (T^3;\M)
\stackrel{a_j} \to \underline{HF}^+_j (T^3;\M)
\stackrel{b_j}\to \underline{HF}^-_j (T^3;\M)
\to\dots.$$ The left two terms have been computed, and we
would like to compute the image of $b_j$. For large $j$,
$\underline{HF}^\infty_j (T^3;\M) \cong \underline{HF}^+_j (T^3;\M)
\cong \Z^2$. By Lemma~\ref{lemma: part1},
$$U\colon\underline{HF}^+_{j+2}(T^3;\M)
\longrightarrow \underline{HF}^+_j(T^3;\M)$$ is an isomorphism for
$j\geq{1\over 2}$ and is injective for $j=-{1\over 2}$. By
$U$-invariance, it follows that $\underline{HF}^\infty_{-{1\over
2}}(T^3;\M)$ maps isomorphically onto the $\Z^2$-summand of
$\underline{HF}^+_{-{1\over 2}}(T^3;\M)$. Hence
$\underline{HF}_{red}(T^3;\M)\cong \Z[t,t^{-1}]$.

Now, $\underline{c}_{red}(\xi_n;\M) \neq 0$ by
\cite[Theorem~4.1]{OSz3}, since $\xi_n$ has a weak filling with
$b_+>0$, whose symplectic form restricts to $[\omega]$ on $T^3$.
\end{proof}

\subsection{Comparison of $E(2)$ and $E(3)$}

We start with a few preparatory lemmas.

\begin{lemma}
There is a symplectic cobordism $W_0$ from $(T^3, \xi_2)$  to $(T^3,
\xi_3)$ which is diffeomorphic to the elliptic surface $E(1)$ with
the tubular neighborhoods of two regular fibers removed.
\end{lemma}

By the above phrase ``from $(T^3,\xi_2)$ to $(T^3,\xi_3)$'' we mean
that $(T^3,\xi_2)$ is the convex boundary and $(T^3,\xi_3)$ is the
concave boundary.

\begin{proof}
Choose an integral basis $([a],[b])$ of $H_1(T^2;\Z)$, where $a$,
$b$ are simple closed curves, and let $\tau_a$, $\tau_b$ be the
positive Dehn twists around $a$ and $b$. Identify $T^3\cong
\R^3/\Z^3=T^2\times (\R/\Z)$.  Let $L$ be the link which is the
union of $a\times\{{i\over 6}\}$ for $i=0,\dots,5$, and
$b\times\{{i\over 6}+{1\over 12}\}$ for $i=0,\dots,5$. We will use
the framing induced from $T^2\times\{t\}$. It is well-known that
$(\tau_a \tau_b)^6=id$. This means that a $(-1)$-surgery on $T^3$
along all the components of the link $L$ yields $T^3$. The
$4$-dimensional cobordism associated to this surgery admits a
Lefschetz fibration over the annulus with generic fiber $F$
diffeomorphic to $T^2$ and twelve singular fibers corresponding to
the twelve Dehn twists in $(\tau_a \tau_b)^6$ --- this is precisely
$E(1)$ minus the neighborhoods of two generic fibers.

If we consider the contact structure $\xi_3$ on $T^3$, we can make
the link $L$ Legendrian with twisting number $0$ (i.e., make each
component a Legendrian divide).  Hence, by applying Legendrian
$(-1)$-surgery along all the components of $L$, we can endow $W_0$
with the structure of a symplectic cobordism. By a direct
computation (see \cite{LS1}) we can show that the Legendrian surgery
along $L$ decreases the Giroux torsion by $1$, so we obtain $\xi_2$
as the convex boundary of the cobordism.
\end{proof}

Let us write the Ozsv\'ath-Szab\'o $4$-manifold invariant of a
closed, oriented $4$-manifold $X$ as:
$$\Phi_X= \sum_{\mathfrak{s}\in \mbox{\tiny Spin}^c(X)}
\Phi_{X,\mathfrak{s}} ~ t^{\mathfrak{s}-\mathfrak{s}_0},$$ where
$\mathfrak{s}_0$ is a reference Spin$^c$-structure.\footnote{Note
that we are using $\mathfrak{s}-\mathfrak{s}_0$ instead of
$c_1(\mathfrak{s})$. This means the statement of Lemma~\ref{lemma:
En} looks slightly different from the results of \cite{OSz5} and
\cite{JM}.} Also given $[\omega]\in H^2(X;\Z)$, we can form:
$$\Phi_{X;\mo}= \sum_{\mathfrak{s}\in \mbox{\tiny Spin}^c(X)}
\Phi_{X,\mathfrak{s}} ~ t^{\langle \omega\cup
(\mathfrak{s}-\mathfrak{s}_0),[X]\rangle}.$$ Here we take
$\mathfrak{s}_0$ to be the canonical Spin$^c$-structure if $\omega$
is symplectic. We then have the following:

\begin{lemma}\label{lemma: En}
Consider the elliptic surfaces $E(n)$, $n\geq 2$.  If $F$ is a
regular fiber, then $$\Phi_{E(n)}= (T-1)^{n-2},$$ where
$T=t^{PD([F])}$. If $\omega$ is a symplectic form on $E(n)$ arising
from the Lefschetz fibration, so that $[\omega]\in H^2(E(n);\Z)$ and
$[\omega|_F]\in H^2(F;\Z)$ is primitive, then
$$\Phi_{E(n);\mo}= (t-1)^{n-2}.$$
\end{lemma}

This was proved for $E(2)$ by Ozsv\'ath-Szab\'o \cite{OSz5} and for
$E(n)$ in general by Jabuka-Mark \cite[Section~1.4.1]{JM}.

\begin{lemma}\label{lemma: b2plus}
Let $W_1$ be $E(2)$ with a $4$-ball and a neighborhood $N(F)$ of a
regular fiber $F$ removed.  Then $b_2^+(W_1)>1$.
\end{lemma}

\begin{proof}
We will produce two pairs of closed surfaces $A_i$ and $B_i$ in
$E(2)$ for $i=1,2$ which are disjoint from some regular fiber $F$
and have intersection form $\begin{pmatrix} 0 & 1 \\1 & m
\end{pmatrix} \oplus
\begin{pmatrix} 0 & 1 \\ 1 & n
\end{pmatrix}$ for some integers $m, n$.
Since both matrices have determinant $-1$, it follows that there is
some linear combination of $A_1$ and $B_1$ and some combination of
$A_2$ and $B_2$ with positive self-intersection.

Recall that $E(2)$ is obtained by gluing two copies of $E(1)-N(F)$.
Take a basis of $H_1(\bdry N(F))= H_1(F\times S^1)$ of the form
$\alpha,\beta,\gamma$, where $\alpha,\beta$ form a basis for
$H_1(F)$ and $\gamma$ is null-homologous in $H_1(N(F))$.  For
$i=1,2$, choose $\alpha_i = \alpha \times \{ \theta_i \}$ and
$\beta_i = \beta \times \{ \theta_i \}$, where $\theta_1 \neq
\theta_2\in S^1$.

By \cite[Lemma 3.1.10]{GS}, $\alpha_1$ and $\beta_2$ bound disjoint
embedded disks in $E(1)-N(F)$; if we think of the Lefschetz
fibration picture, these disks are ``thimbles'' which correspond to
vanishing cycles. Since we have disks on both copies of $E(1)-N(F)$,
they glue to give two disjoint spheres which we call $B_1$ and
$B_2$.  On the other hand, let $A_1=\beta \times \gamma$ and $A_2=
\alpha \times \gamma$.  Then we have $A_i \cdot B_i = \pm 1$.
Moreover, if $i\not=j$, then $A_i\cdot B_j=0$. It is also clear that
$A_i \cdot A_j =0$ because $A_i$ can be pushed off of $N(F)$.
Finally, the $A_i$ and $B_i$ can be made disjoint from some copy of
$F$.
\end{proof}

We are now ready to prove Theorem~\ref{thm: polynomial}.

\begin{proof}[Proof of Theorem~\ref{thm: polynomial}]
We compare the $4$-dimensional Ozsv\'ath--Szab\'o invariants of the
elliptic surfaces $E(2)$ and $E(3)$.

We regard $E(2)$ as a cobordism $W=E(2)-B^4-B^4$ from $S^3$ to
$S^3$, and decompose $W$ into $W_1= E(2)-B^4-N(F)$ and
$W_2=N(F)-B^4$ along a $3$-torus $T^3$. Let $\omega$ be a symplectic
form on $E(2)$ arising from the Lefschetz fibration, so that
$[\omega]\in H^2(E(2);\Z)$ and $[\omega|_{T^3}]\in H^2(T^3;\Z)$ is
primitive.

Let $\Theta_-$ be the top degree generator of
$\underline{HF}^-(S^3;\mo)$. If $\mathfrak{s}$ is a Spin$^c$
structure on $E(2)$, then, by the composition law (Theorem~\ref{thm:
comp law}), we have:
\begin{equation} \label{eqn: cobord}
\sum_{\eta\in H^1(T^3;\Z)} \Phi_{E(2),\mathfrak{s}+\delta
\eta}~t^{\langle \omega\cup
(\mathfrak{s}-\mathfrak{s}_0+\delta\eta),[E(2)]\rangle}
=\underline{F}^+_{W_2,\mathfrak{s}|_{W_2};\mo} \circ
\underline{F}^{mix}_{W_1,\mathfrak{s}|_{W_1};\mo}(\Theta_-).
\end{equation}
Observe that $\underline{F}^{mix}_{W_1;\mo}$ is defined, since
$b_2^+(W_1)>1$ by Lemma~\ref{lemma: b2plus}. Here $\delta\colon
H^1(T^3) \to H^2(E(2))$ is the connecting homomorphism in the
Mayer-Vietoris sequence for $E(2)$ involving $W_1$ and $W_2$. The
map $\delta$ is equivalent, via Poincar\'e Duality, to the inclusion
map $i\colon H_2(T^3)\to H_2(E(2))$. This means that $PD([F])$ is in
the image of $\delta$.  Now, if we sum Equation~\ref{eqn: cobord}
over a suitable collection of Spin$^c$ structures, we obtain:
\begin{equation} \label{eqn: E2}
\Phi_{E(2);\mo}=\underline{F}^+_{W_2;\mo} \circ
\underline{F}^{mix}_{W_1;\mo}(\Theta_-)=1,
\end{equation}
by Lemma~\ref{lemma: En}.

Now consider the map
$$\underline{F}^+_{W_2;\mo}\colon
\underline{HF}^+(T^3;\mo)\to
\underline{HF}^+(S^3;\mo)=HF^+(S^3;\Z)\otimes_{\Z} \mo.$$ Observe
that all the summands of $HF^+(S^3;\Z)\otimes_{\Z} \mo$ are
$\Z[\R]$, whereas only one summand of $\underline{HF}^+(T^3;\mo)$ is
$\Z[\R]$ by Lemma~\ref{prop: summary}, and the rest have torsion.
Since $\underline{F}^+_{W_2;\mo}$ is $t^a$-invariant, the only
summand of $\underline{HF}^+(T^3;\mo)$ which maps nontrivially to
$\underline{HF}^+(S^3;\mo)$ is $\Z[\R]$. This means that
$\underline{F}^{mix}_{W_1;\mo}$ maps $\Theta_-$ to an element with
nonzero projection to the summand $\Z[\R]\subset
\underline{HF}^+(T^3;\mo)$, which, in turn, is mapped to the bottom
generator $\Theta_+$ of $\Z[\R]\subset \underline{HF}^+(S^3;\mo)$.
Recall that $\underline{c}(T^3,\xi_n;\mo)$ also has a nontrivial
projection to $\Z[\R]\subset \underline{HF}^+(-T^3;\mo)$.

Next we regard $E(3)$ as a cobordism $W'=E(3)-B^4-B^4$, which is
decomposed into $W_1$, $W_0$, and $W_2$. Orient the $3$-torus
$M_1=\bdry W_1 \cap \bdry W_0$ as $\bdry W_1$ and the $3$-torus
$M_2=\bdry W_0\cap \bdry W_2$ as $\bdry W_0$. We view the cobordism
$W_0$ turned ``upside-down'' so that the induced map
$$\underline{F}^+_{W_0;\mo}\colon \underline{HF}^+(M_1;\mo)\to
\underline{HF}^+(M_2;\mo)$$ sends
$$\underline{c}(-M_1,\xi_2;\mo)\mapsto \underline{c}(-M_2,\xi_3;\mo).$$
Restricted to the summand $\Z[\R]$, this means that
$$\underline{F}^+_{W_0;\mo}\colon\Z[\R]\to
\Z[\R]$$ is multiplication by $p(t)$. As in the $E(2)$ case, we
compute:
\begin{equation}
\Phi_{E(3);\mo} = \underline{F}^+_{W_2;\mo} \circ
\underline{F}^+_{W_0;\mo} \circ
\underline{F}^{mix}_{W_1;\mo}(\Theta_-) =t-1.
\end{equation}
Since the projection of $\underline{F}^{mix}_{W_1;\mo}(\Theta_-)$ to
$\Z[\R]$ is nonzero, we have that $p(t)=t-1$.
\end{proof}

\s\n {\em Acknowledgements.} We thank Michael Hutchings and Peter
Ozsv\'ath for discussions which were helpful in determining $p(t)$.
Part of this work was done while the second author visited the
University of Tokyo. He would like to thank Takashi Tsuboi and the
University of Tokyo for their hospitality.

\end{document}